\documentclass[12pt,reqno]{amsart}
 \usepackage[usenames,dvipsnames]{color}
\usepackage{graphicx}
\usepackage{fullpage}
\usepackage{amsthm}
\usepackage{shuffle}
\usepackage{float}
\usepackage[mathscr]{euscript}
\usepackage[colorlinks=true,
   citecolor=cyan,urlcolor=blue]{hyperref}

\usepackage[french]{todonotes}

\usepackage{wrapfig}
\usepackage{amssymb}
\usepackage{amsmath}

\def\rouge{\textcolor{red}}
\def\bleu{\textcolor{blue}}

\makeatletter
\newtheorem*{rep@theorem}{\rep@title}
\newcommand{\newreptheorem}[2]{
\newenvironment{rep#1}[1]{
\def\rep@title{#2 \ref{##1}}
\begin{rep@theorem}}
{\end{rep@theorem}}}
\makeatother
\def\auteur#1{{\sc #1}}
\def\titreref#1{{\em #1}}
\def\vol#1{{\bf #1}}

\newtheorem{question}{\bleu{Question.\!\!}}

\newtheorem{proposition}{\bleu{Proposition}}

\newreptheorem{thm}{Theorem}

\numberwithin{equation}{section}
\numberwithin{thm}{section}
\numberwithin{lemma}{section}
\numberwithin{rmk}{section}
\numberwithin{prop}{section}
\numberwithin{cor}{section}
\numberwithin{defn}{section}
\numberwithin{alg}{section}

\def\mbf#1{{\mathbf #1}}

\newcommand{\stirling}[2]{\genfrac{\{}{\}}{0pt}{}{#1}{#2}}

\renewcommand{\S}{\mathbb{S}}
\newcommand{\comp}{\operatorname{comp}}

\newcommand{\area}{\operatorname{area}}

\newcommand{\dinv}{\operatorname{dinv}}

\newcommand{\rank}{\operatorname{rank}}

\def\define#1{\bleu{\bf{#1}}}

\def\X{\mathbf{x}}

\def\N{\mathbb{N}}

\def\pref#1{{\rm (\ref{#1})}}
\def\qref#1{\bleu{Question-}\ref{#1}}

\DeclareMathOperator{\D}{\mathbf{D}}

\def\Z{\mathbb{Z}}
\def\Q{\mathbb{Q}}

\def\charac{\raise 2pt\hbox{\large$\chi$}}
\def\Id{\mathrm{Id}}

\newdimen\carrelength
\def\jcarre{\jaune{\linethickness{\carrelength}\line(1,0){.85}}}
\def\bleu{\textcolor{blue}}
\def\rouge{\textcolor{red}}
\def\jaune{\textcolor{yellow}}
\def\bleupale#1{{\color{SkyBlue} #1}}

\def\Dyck#1#2{\mathscr{D}_{#1,#2}}
\def\Cat#1#2{{C}_{#1,#2}}

\title[Open Questions]{Open Questions for operators related to Rectangular Catalan Combinatorics}
\author{F.~Bergeron}
\address{\href{http://bergeron.math.uqam.ca}{D\'epartement de Math\'ematiques, Lacim, UQAM.}}
  \email{\href{mailto:bergeron.francois@uqam.ca}{bergeron.francois@uqam.ca}}
  \date{\rouge{\bf March 20, 2016}. This work was supported by NSERC}

\begin{document}


\begin{abstract}
We formulate many open questions regarding the Schur positivity of the effect of interesting operators on symmetric functions, and give supporting evidence for why one should expect such behavior.
\end{abstract}

\maketitle
 \parskip=0pt
{ \setcounter{tocdepth}{1}\parskip=0pt\footnotesize \tableofcontents}
\parskip=8pt  
\parindent=20pt


\section*{Introduction} 
The aim of this text is to present, in a concise manner, a set of open questions relating to operators on symmetric functions that are relevant to rectangular Catalan combinatorics. In some form or another, some of these questions have already been considered, but we thought it good to have them all state together. For sure, several questions are new. On top of giving proofs for several of their specializations, the identities and properties considered here have been thoroughly checked for as large as possible a set of special cases. Hence they are stated with a good degree of confidence.

The effect of these operators appears to be elegantly linked to the combinatorics of rectangular Dyck paths, and associated parking functions. This is explicitly made evident when one specializes one of the parameters, say $t$, to be equal to $1$.  


\section{Elliptic Hall Algebra} \label{sec:elliptic}
In a fashion somewhat similar to how creation operators are used in quantum mechanics, the main actors of our story are operators on symmetric functions that we eventually apply to the simplest symmetric function $1$, aiming at constructing interesting symmetric functions. These operators belong to a realization of the ``positive part'' $\mathcal{E}$ of the ``elliptic Hall algebra''  (see below for more details) as a subalgebra of ${\mathrm{End}}(\Lambda)$, where  
     $$\Lambda=\sum_{d\geq o} \Lambda_d,$$ 
 is the degree-graded algebra of  symmetric functions over the field $\Q(q,t)$. In this context, they are generated by two families of ``well-known'' operators. The first of these is the set of operators that correspond to multiplication by symmetric functions. 
\begin{equation}\label{multiplicateurs}
    (-)\cdot f:\Lambda_d\longrightarrow\Lambda_{d+k},\qquad \hbox{(that is  $g\mapsto g\cdot f$, for $f\in\Lambda_k$)};\end{equation} 
while the second is a family  $\{\D_n\}_{n\in\Z}$ of operators considered (see~\cite{IdPosCon}) in the study of Macdonald polynomials. 
Let us recall that these operators $\D_n$, send degree $d$ symmetric function to degree $d+n$
  \begin{displaymath} \bleu{\D_k:\Lambda_{d} \longrightarrow \Lambda_{d+k} }.\end{displaymath}
They are jointly defined by
the generating function equality  
        \begin{equation}\label{defDn}
          \bleu{\sum_{k=-\infty}^{\infty} \D_k( g(\mbf{x}))\, z^k := g \!\left[\mbf{x} +M/{z}\right]  \sum_{n\geq 0} e_n(\mbf{x})\,(-z)^n},
      \end{equation}
here written using plethystic notation (see~\cite{IdPosCon} for more on this), for any $g(\mbf{x})\in\Lambda_d$, and writing $M=M(q,t)$ for $(1-t)(1-q)$. It may be shown that $\D_0$ is a \define{Macdonald eigenoperator}. This is to say that it affords the (combinatorial) Macdonald polynomials as joint eigenfunctions.
It may also be worth recalling that, for all $k$, we have
   	\begin{equation}
	    \bleu{\D_{k+1}=\frac{1}{M} [\D_k,p_1]},	
\end{equation}	
with $[-,-]$ standing for the usual Lie bracket of operators, and $e_1$ correspond to multiplication by the degree $1$ elementary symmetric function. In other words, all of the $\D_k$ (for $k>0$) are obtained as order $k$ Lie-derivatives, with respect to the operator of multiplication by $p_1/M$. Maybe even better for calculation purposes, we have
 \begin{equation}
     \bleu{\D_{k+j}=\frac{1}{(1-t^j)(1-q^j)} [\D_k,p_j]},
 \end{equation}
 for all $k$ and $j$. Indeed, considering the above operator generating series $\mathcal{D}(z)= \sum_{k=-\infty}^{\infty} \D_k \,z^k$, one may check that
 $$z^j [\mathcal{D}(z),p_j(\mbf{x})]=  (1-t^j)(1-q^j)\,\mathcal{D}(z),$$
simply by calculating that\footnote{Recall that, in plethystic notation, one has   $p_j[\mbf{x}+M/z] =p_j(\mbf{x})+(1-t^j)(1-q^j)/z^j$.}   \begin{eqnarray*}
     z^j  [\mathcal{D}(z),p_j]\, g(\mbf{x}) &=& \sum_{k=-\infty}^{\infty} [\D_k,p_j]( g(\mbf{x}))\, z^k\\
                   &=&z^j\Big(g \!\left[\mbf{x} +M/{z}\right]  p_j \!\left[\mbf{x} +M/{z}\right]-g \!\left[\mbf{x} +M/{z}\right] p_j(\mbf{x})  \Big) 
                             \sum_{n\geq 0} e_n(\mbf{x})\,(-z)^n\\                   
                   &=&z^j\Big(g \!\left[\mbf{x} +M/{z}\right]  (1-t^j)(1-q^j)/z^j  \Big) 
                             \sum_{n\geq 0} e_n(\mbf{x})\,(-z)^n\\                   
                   &=& (1-t^j)(1-q^j)\,\mathcal{D}(z) \, g(\mbf{x}).                  
  \end{eqnarray*}
  In particular, we get $(1-t^j)(1-q^j) D_j= [D_0,p_j]$, reducing the calculation of $D_j$ to that of $D_0$, modulo a single bracket operation.

We recall from \cite{SchiffVassMac} that the positive part $\mathcal{E}$ of the elliptic Hall algebra\footnote{The full algebra is $\Z^2$-graded, but we only need the positive components for our purpose.}   may be realized as a $(\N\times\N)$-graded algebra of operators on $\Lambda$
  $$\bleu{\mathcal{E}=\bigoplus_{(m,n)\in\N^2}\mathcal{E}_{m,n}},$$
with the operators in the homogeneous component $\mathcal{E}_{m,n}$ sending $\Lambda_d$ to $\Lambda_{d+n}$. Special cases of these operators were introduced in~\cite[see Thm 4.4]{IdPosCon}, where relevant properties were also stressed out.
As mentioned previously, generators for $\mathcal{E}$ include the $D_k$ operators, which are considered to be of degree $(1,k)$, as well as the operators of multiplication by symmetric functions lying in $\Lambda_d$, which are considered to be of degree $(0,d)$.
It is also established in \cite{SchiffVassMac} that, for each pair of coprime integers $(a,b)$, 
there are ring monomorphisms  
   $$\bleu{\Theta_{a,b}:\Lambda\longrightarrow\mathcal{E}},$$
explicitly described below, such that $\Theta_{a,b}(\Lambda_d))\subseteq \mathcal{E}_{(ad,bd)}$. In particular, this says that one has commutation of operators belonging to the image of $\Theta_{a,b}$, for any given coprime pair $(a,b)$.
Following~\cite{GorskyNegut}, the easiest way to give an explicit description of these monomorphisms, is to fix the $\Theta_{a,b}$-image of the functions
\begin{equation}\label{defn_qd}
   \bleu{q_d=q_d(\mathbf{x};q,t):=\sum_{j+k=d-1} \left(-qt\right)^{-j} s_{(j\,|\,k)}(\mathbf{x})},
   \end{equation}
 where $s_{(j\,|\,k)}(\mathbf{x})$ stands for the hook indexed Schur symmetric functions, where the hook has one part of size $j+1$, and $k$ parts of $1$.	 
Hence, we clearly have
\begin{eqnarray*}
q_{{1}}(\mbf{x}) &=&s_{{1}}(\mbf{x}), \\
q_{{2}}(\mbf{x}) &=& s_{{11}}  (\mbf{x}) -\frac{1}{q t }\, s_{2}(\mbf{x}),\\
q_{{3}}(\mbf{x}) &=& s_{111} (\mbf{x}) -\frac{1}{q t }\, s_{21}(\mbf{x}) +\frac{1}{q^2t^2}\, s_{3}(\mbf{x}),\\
q_{{4}}(\mbf{x}) &=& s_{1111}(\mbf{x}) -\frac{1}{q t }\, s_{211}(\mbf{x})+\frac{1}{q^2t^2}\, s_{31}(\mbf{x})-\frac{1}{q^3t^3}\, s_{4}(\mbf{x}),\\
q_{{5}}(\mbf{x}) &=&s_{{11111}}  (\mbf{x}) -\frac{1}{q t }\, s_{2111}(\mbf{x})+\frac{1}{q^2t^2}\, s_{311}(\mbf{x})-\frac{1}{q^3t^3}\, s_{41}(\mbf{x})+\frac{1}{q^4t^4}\, s_{5}(\mbf{x}). \end{eqnarray*}
Observe that, when the parameters $q$ and $t$ are such that $qt=1$, then $q_d(\mbf{x})$ specialize to the classical power sum symmetric functions $p_d(\mbf{x})$. Hence the set $\{q_d((\mbf{x})\}_d$ forms an independent algebraic generator set for $\Lambda$.
It will be useful to consider the linear basis of $\Lambda_d$ which is made up of the functions
      $$\bleu{q_\mu(\mathbf{x}):=q_{\mu_1}(\mathbf{x})q_{\mu_2}(\mathbf{x})\cdots q_{\mu_k}(\mathbf{x})},$$
with $\mu=\mu_1\mu_2\cdots\mu_k$ running over the set of integer partitions of $d$. 

For each degree $d$ symmetric function $f_d$, and any coprime pair $(a,b)$, we consider operators  $\Theta_{a,b}(f_d)\in\mathcal{E}_{ad,bd}$.
In other words, assuming that we have described\footnote{See formula~\pref{def_theta}.}  $\Theta_{a,b}(q_d)$, and that the expansion of $f_d$ in the basis $\{q_\mu(\mbf{x})\}_{\mu\vdash d}$ is 
     $$\bleu{f_d(\mathbf{x})=\sum_{\mu\vdash d} c_\mu(q,t)\, q_\mu(\mathbf{x})},$$
then we clearly have
\begin{equation}\label{def_oper}
   \bleu{\Theta_{a,b}(f_d):=\sum_{\mu\vdash d} c_\mu(q,t)\, \Theta_{a,b}(q_\mu)},\qquad \hbox{with}\qquad \bleu{\Theta_{a,b}(q_\mu)=\Theta_{a,b}(q_{\mu_1})\cdots \Theta_{a,b}(q_{\mu_k})},
\end{equation}
observe that we do not worry here about the order in which the operators $\Theta_{a,b}(q_{\mu_i})$ should be applied, since they commute. This fact is assured by the general properties of the elliptic Hall algebra established in \cite{SchiffVassMac} and related papers. These properties insure that all the construction described here make sense. One of the striking implications of the properties of $\mathcal{E}$, see~\cite{compositionalshuffle}, is that, for all $(a,b)$ coprime, all $d\in\N$, and all $f_d$, one has the operator identity
\begin{equation}\label{nabla}
    \bleu{\nabla\,\Theta_{a,b}(f_d)\, \nabla^{-1}=  \Theta_{a+b,b}(f_d)}.
\end{equation}
Hence, using the inverse relation, 
 \begin{equation}\label{nabla_inv}
    \bleu{\nabla^{-1}\,\Theta_{a,b}(f_d)\, \nabla=  \Theta_{a-b,b}(f_d)},
\end{equation}
one may extend $\Theta_{a,b}$ to negative values of $a$. It may then be checked that
\begin{equation}
   \bleu{(-qt)\,\omega^*\ \Theta_{a,b}(f_d)\ \omega^* = \Theta_{-a,b}(\omega^*\,f_d)}.
 \end{equation}
Here $\nabla$ stands for the much-discussed Macdonald eigenoperator, which is such that $\nabla(e_n)$ gives the bigraded Frobenius characteristic of the diagonal coinvariant space of $\S_n$ (see \cite{bergeron} for more on this). We also denoted by $\omega^*$ is the involutive operator that sends $f_d(\mbf{x};q,t)$ to 
  	$$(-1/qt)^{d-1}\,\omega f_d(\mathbf{x};1/q,1/t).$$
Because of ties with representation theory, we are interested in functions $f_d$ such that the application of the operators $\Theta_{a,b}(f_d)$ to the constant symmetric function $1$ gives Schur-positive\footnote{Recall that this means that its Schur function expansion has coefficients in $\N[q,t]$. See Appendix~\ref{appendixA}.} symmetric function, for any coprime $a,b\geq 1$. When this is so, we say that $f_d$ gives rise to \define{Schur-positive operators}, and denote by $f_d^{(a,b)}(\mbf{x};q,t)$ the symmetric function $\Theta_{a,b}(f_d)(1)$.

\subsection*{Definition of the basic operators}\label{def_qd} For coprime $a,b\geq 1$, we now describe how to recursively construct\footnote{This essentially comes from \cite{GorskyNegut}, which in turn is a translation of the results presented in \cite{SchiffVassMac} and related papers.}  basic operators $\Theta_{a,b}(q_d)$ as degree $(ad,bd)$ (non commutative) polynomials in the $\D_0$ and multiplication by $e_1$, respectively considered to be of degree $(1,0)$ and $(0,1)$. We start by writing $Q_{0k}$ for the operator of multiplication by the symmetric function $q_k(\mbf{x})$, and $Q_{k0}$ for the operator $(-1)^k\,D_k$. For $m,n\geq 1$, we then recursively define operators $Q_{mn}$ by the Lie bracket formula
    \begin{equation}
         \bleu{Q_{mn}:=\frac{1}{M}\, [Q_{uv},Q_{kl}]},
     \end{equation}
where we choose $(k,l)$ such that $(m,n)=(k,l)+(u,v)$, with $(k,l)$ and $(u,v)$ lying in $\N^2$,  $l-(k n/m)$ is minimal, and such that
    $$\det\begin{pmatrix} u & v\\ k& l\end{pmatrix}=d,$$
 with $d$ standing for the greatest common divisor of $m$ and $n$. 
 Moreover, if $(m,n)=(ad,bd)$, we ask that $(k,l)$ be chosen to be the same as it would for $(a,b)$. 
 For example, we get the Lie bracket expressions
   $$Q_{43}=\frac{1}{M^6} [[e_{{1}},D_{{0}}],[[e_{{1}},D_{{0}}],[[e_{{1}},D_{{0}}],D_{{0}}]]],$$
 or
  $$Q_{63}=\frac{1}{M^8}[[e_{{1}},D_{{0}}],[[[e_{{1}},D_{{0}}],D_{{0}}],[[[e_{{1}},D_{{0}}],D_
{{0}}],D_{{0}}]]].$$
For sure, the monomials that occur in the expansion of $Q_{mn}$ involve $m$ copies of $D_0$, and $n$ copies of $e_1$. We these operators at hand, we my now define the monomorphisms $\Theta_{a,b}$ by setting
    \begin{equation}\label{def_theta}
        \bleu{\Theta_{a,b}(q_d):=Q_{ad,bd}}.
     \end{equation}
 In particular, our above notation convention makes it so that we may write $q_d^{(a,b)}(\mbf{x};q,t)$ for $Q_{ad,bd}(1)$.

\section{Schur Positive Operators}
More generally, we are interested in symmetric functions $f_d$, here called \define{seeds}, such that we get Schur-positive operators $\Theta_{a,b}(f_d)$. Recall that this means that the symmetric function $f_d^{(a,b)}(\mbf{x};q,t)$ (which is just another way of writing $\Theta_{a,b}(f_d)(1)$) expands with coefficients in $\N[q,t]$ in the Schur function basis. In other terms, we want 
    \begin{equation}\label{condition_schur}
        \bleu{ 0\preceq_sf_d^{(a,b)}(\mbf{x};q,t)}.
     \end{equation}
Chief among the interesting operators of this kind are those for which the seed $f_d(\mbf{x})$ is chosen to be
\begin{enumerate}
\item the elementary symmetric functions $e_d(\mbf{x})$;
\item the suitably normalized complete homogeneous symmetric functions $(-qt)^{1-k}h_d(\mbf{x})$;
\item or more generally, the renormalized Schur functions $({-qt})^{-\iota(\mu)}s_\mu(\mbf{x})$
where, for an integer partition $\mu$ of $d$, we set
      $$\bleu{\iota(\mu):=\sum_{i=1}^{\ell(\mu)} \charac(\mu(i)-i)};$$
 \item for any partition $\mu$, the monomial symmetric functions  $(-1)^{n-\ell(\mu)}\,m_\mu(\mbf{x})$.
\end{enumerate}
Observe that all the above operators coincide for $d=1$, since there is essentially but one seed of degree $1$, up to a constant factor. In other words, the operators only become different when one considers $d>1$.
It immediately follows from the definitions that the resulting operators are linked by the same relations as those  between their seeds. 
\subsection*{The compositional $(ad,bd)$-shuffle conjecture}
  Recall from~\cite{compositionalshuffle} the conjectured combinatorial formula for the effect on $1$ of the operators having as seed the symmetric function $C_\alpha:=\mathbf{C}_\alpha(1)$, where one sets $\mathbf{C}_\alpha:= \mathbf{C}_{a_1}\mathbf{C}_{a_2}\cdots \mathbf{C}_{a_\ell}$, for any composition $\alpha=a_1a_2\ldots a_\ell$ of $d$, with the individual operators $\mathbf{C}_a$ specified by the formula
  $$\bleu{\mathbf{C}_a f[\X] := (-t)^{1-a} f\!\left[\X-{(t-1)}/{(tz)}\right]\, \sum_{m \geq 0} z^m h_m[\X]\, \Big|_{z^a}},$$   
 where $(-)\big|_{z^a}$ means that we take the coefficient of $z^a$ in the series considered.  
In our current notations  the compositional $(ad,bd)$-shuffle conjecture (of ~\cite{compositionalshuffle}) states that  
\begin{equation}\label{shuffle_conjecture}
 \bleu{C_\alpha^{(a,b)}(\mbf{x};q,t) =\sum_\gamma q^{\area(\gamma)} \sum_{\pi} t^{\dinv(\pi)} s_{\comp(\pi)}(\mathbf{x})},
 \end{equation}
where the first sum is over all $(ad,bd)$-Dyck paths that return to the diagonal at positions specified by the composition $\alpha$, and the second is over parking functions whose underlying path is $\gamma$ (Necessary concepts and notations are defined in the appendix). It is known (see~\cite{grojnowski}) that for any given $(ad,bd)$-Dyck path $\gamma$, the summation 
     $$\bleu{\sum_{\pi} t^{\dinv(\pi)} s_{\comp(\pi)}(\mathbf{x})},$$
is a LLT-polynomial, which is Schur positive, hence it follows from \pref{shuffle_conjecture}  that $C_\alpha^{(a,b)}(\mbf{x};q,t)$ is Schur positive, that is
\begin{equation}
     \bleu{0\preceq_s C_\alpha^{(a,b)}(\mbf{x};q,t)}.
\end{equation}
 It is also known that this LLT polynomial specializes to $e_{\rho(\gamma)}(\mbf{x})$ at $t=1$, so that proving \pref{shuffle_conjecture}  would also show that, for all $\alpha$ and all coprime $(a,b)$, 
\begin{equation}
     \bleu{0\leq_{\rouge{e}} C_\alpha^{(a,b)}(\mbf{x};q,1)}.
\end{equation}
In a similar vein, our first open question is:

\begin{question}\label{question_schur}
Can we prove that 
\begin{equation}\label{inegalite_schur}
  \bleu{0\preceq_s (-{qt})^{-\iota(\mu)}s_\mu^{(a,b)}(\mbf{x},q,t)},
\end{equation}
 for all partition $\mu$ of $d$, and all coprime $a,b\geq 1$? Can we explain this in terms of bigraded subrepresentations\footnote{For a clearer statement concerning this, see Section~\ref{section_incl}.} of the $\S_n$-module of generalized diagonal harmonics?
\end{question}
\noindent For example, we have
\begin{eqnarray*}
(-qt)^{-2}s_{{3}}^{(1,2)}(\mathbf{x};q,t)&=& \left( q+t \right) s_{{222}}+s_{{321}}+ \left( {q}^{2}+qt+{t}^{2}+q+t \right) s_{{2211}}\\
&&\qquad + \left( q+t \right) s_{3111}+ \left( q+t \right)  \left( {q}^{2}+t^{2}+q+t \right) s_{{21111}}\\
&&\qquad +\left( {q}^{4}+{q}^{3}t+{q}^{2}{t}^{2}+q{t}^{3}+{t}^{4}+{q}^{2}t+q{t}
^{2} \right) s_{{111111}}
,\\
(-qt)^{-1} s_{{21}}^{(1,2)}(\mathbf{x};q,t)&=& \left( {q}^{2}+qt+{t}^{2} \right) s_{{222}}+ \left( q+t \right)  \left( {q}^{2}+{t}^{2}+q+t \right) s_{{2211}}\\
&&\qquad + \left( q+t \right) s_{{3,21}}+ \left( {q}^{2}+qt+{t}^{2} \right) s_{{3111}}\\
&&\qquad+ \left( {q}^
{4}+{q}^{3}t+{q}^{2}{t}^{2}+q{t}^{3}+{t}^{4}+{q}^{3}+2\,{q}^{2}t+2\,q{
t}^{2}+{t}^{3} \right) s_{{21111}}\\
&&\qquad+ \left( {q}^{2}+qt+{t}^{2}
 \right)  \left( {q}^{3}+{t}^{3}+qt \right) s_{{111111}}.
\end{eqnarray*}
Answering \qref{question_schur} in the positive would settle many previous conjectures.
The special case $\mu=1^d$ (which coincides for both this open question and the one below) corresponds to the known Schur positivity of $\nabla(e_n)$; and for general $b=1$, it is implied by the Shuffle Conjecture~\cite[Conjecture 3.1]{HHLRU}. For $\mu=(d)$, it corresponds to a special case of \cite[Conjecture 3.3]{HMZ}.  For general $\mu$, and $b=1$, it corresponds to \cite[Conjecture I]{IdPosCon}. Indeed, this last assertion follows from~\pref{nabla}. To see it, we apply the operators of identity \pref{nabla} to the constant symmetric function $1$,  to get
\begin{eqnarray}\label{nabla_fonct}
     \bleu{f_d^{(a+b,b)}(\mbf{x};q,t)} &=&{\Theta_{a+b,b}(f_d)(1)}\nonumber\\
          &=&{\nabla\,\Theta_{a,b}(f_d)\nabla^{-1}(1)}\nonumber\\
          &=&{\nabla\,\Theta_{a,b}(f_d)(1)}\nonumber\\
          &=&\bleu{\nabla f_d^{(a,b)}(\mbf{x};q,t)},
 \end{eqnarray}
 for all $(a,b)$ and any seed $f_d$. Then, using the fact that $\Theta_{01}=\Id_\Lambda$, one needs only choose $(a,b)=(0,1)$ and $f_d=({-qt})^{-\iota(\mu)}s_\mu$ to get back the relevant conjecture. It follows that
    \begin{equation}\label{nabla_r_f}
        \bleu{f_d^{(r,1)}(\mbf{x};q,t)=\nabla^r(f_d)}.
   \end{equation}
 and in particular, that
    \begin{equation}\label{nabla_r_en}
        \bleu{e_d^{(r,1)}(\mbf{x};q,t)=\nabla^r(e_d)}.
   \end{equation}
Clearly, if a seed $f_d$ expands positively in the basis
      $$\bleu{{(-qt})^{-\iota(\mu)}s_\mu(\mathbf{x};q,t)},$$
  then the associated $f_d^{(a,b)}(\mathbf{x};q,t)$ will perforce be $s$-positive if \qref{question_schur} is answered positively. Obvious cases include $e_d(\mbf{x})$, $(-qt)^{1-d}h_d(\mbf{x})$, $q_d(\mathbf{x})$ (in view of ~\pref{defn_qd}), and
   $$\bleu{(-1)^{k-1} p_k(\mathbf{x})=\sum_{j+k=d-1} (qt)^j\left(\left(-qt\right)^{-j} s_{(j\,|\,k)}(\mathbf{x})\right)}.$$
Another general family of cases goes as follows,
\begin{question}\label{question_monomiales}
Can we prove that  
\begin{equation}\label{inegalite_monomiale}
   \bleu{0\preceq_s (-1)^{d-\ell(\mu)}\, m_\mu^{(a,b)}(\mbf{x};q,t)},
  \end{equation}
for all partition $\mu$ of $d$, and all coprime $a,b\geq 1$?
\end{question}
\noindent Preferably, this would be explained by introducing adequate bigraded $\S_n$-modules whose bigraded Forbenius characteristic would correspond to these Schur positive expressions. For example, we have
\begin{eqnarray*}
-m_{{21}}^{(1,1)}(\mathbf{x};q,t)&=&  2\,s_{{3}}+ \left( {q}^{2}t+q{t}^{2}+2\,{q}^{2}+2\,qt+2\,{t}^{2}+2\,q+
2\,t \right) s_{{21}}\\
&&\qquad + \left( {q}^{3}t+{q}^{2}{t}^{2}+q{t}^{3}+2\,{q}
^{3}+2\,{q}^{2}t+2\,q{t}^{2}+2\,{t}^{3}+2\,qt \right) s_{{111}}.
\end{eqnarray*}
Once again \qref{question_monomiales} relates to previous conjectures. 
For instance, the case $b=1$ corresponds to Conjecture IV of~\cite{IdPosCon}, which asserts the Schur-positivity of $\nabla^a((-1)^{d-\ell(\mu)}\, m_\mu)$.
For both Inequalities~\pref{inegalite_schur} and \pref{inegalite_monomiale}, we have explicitly checked by explicit computer algebra calculations that we do indeed have Schur-positivity for all possible cases of $\mu\vdash d$ with $1\leq ad,bd\leq 12$.


\section{Schur Inclusions}\label{section_incl}

The following considerations (greatly) extend the second observation of \cite[Conjecture III]{IdPosCon}.  
We now consider $s$-positive difference of operators. From the point of view of representation theory, this corresponds to inclusion of graded $\S_n$-modules. For our current purpose, it is convenient to denote by $f_{m,n}(\mbf{x};q,t)$  the symmetric function $f_d^{(a,b)}(\mbf{x};q,t)$, when $(m,n)=(ad,bd)$ and $d=\gcd(m,n)$. Then, let us write $f_{m,n}(\mbf{x};q,t)\preceq_s g_{m,n}(\mbf{x};q,t)$, if and only if the difference $g_{m,n}(\mbf{x};q,t)-f_{m,n}(\mbf{x};q,t)$ is Schur-positive. Our first  observation\footnote{Experimentally supported by calculating all cases with $m,n\leq 9$.} is that 
\begin{equation}\label{S_Inclusion}
   \bleu{q^\alpha\,e_{m-1,n}(\mbf{x};q,t)\ \preceq_s\  e_{m,n}(\mbf{x};q,t)},
 \end{equation}
where $\alpha=\alpha(m,n)$ is the number of cells between the corresponding staircase paths (see \pref{defstaircase} for the definition of the $(m,n)$-staircase path). At $t=1$, we may explain combinatorially that
 \begin{equation}\label{E_Inclusion}
   \bleu{e_{m,n}(\mbf{x};q,1)-q^\alpha\,e_{m-1,n}(\mbf{x};q,1)\in\N[q][e_1,e_2,\ldots]},
 \end{equation}
since the difference between the right-hand side and left-hand side corresponds to a weighted enumeration of the $(m,n)$-Dyck paths that
cannot be obtained from $(m-1,n)$-Dyck paths by the simple addition of a final horizontal step.
On the other hand, the Schur-positivity of \pref{S_Inclusion} is surprising, since suggests that there is some
``dinv'' weight-correcting injection between $(m-1,n)$-Dyck paths and $(m,n)$-Dyck paths. Such a correction seems far from obvious. 

To state our next observed property, we need to introduce the following linear operator. For a partition $\mu$, let us denote by $\overline{\mu}$ the partition obtained by removing the first column of $\mu$. Then, we set $\overline{s_\mu(\mbf{x})}:=s_{\overline{\mu}}(\mbf{x})$, and extend linearly to all symmetric functions. We this notation at hand, we have observed that, similarly to \pref{E_Inclusion}, we have
\begin{equation}\label{Inclusionco}
   \bleu{q^{\alpha'}\,\overline{e_{m,n-1}(\mbf{x};q,t)}\ \preceq_s\  \overline{e_{m,n}(\mbf{x};q,t)}}.
 \end{equation}
 In this case, much as before, $\alpha'=\alpha'(m,n)$ is the number of integer points between the  $(m,n-1)$-staircase path and the minimal $(m,n)$-staircase.
 For example, we have
  \begin{eqnarray*}
       \overline{e_{4,6}(\mbf{x};q,t)}-q^2\,\overline{e_{4,5}(\mbf{x};q,t)} &=&
(q{t}^{7}+{t}^{8}+{q}^{2}{t}^{5}+q{t}^{6}+{q}^{4}{t}^{2}+{q}^{3}{t}^{3}+2\,{q}^{2}{t}^{4}+q{t}^{5})\,s_0(\mbf{x})\\
&&+t \left( q+t \right)  \left( {t}^{5}+q{t}^{3}+{t}^{4}+{q}^{3}+{q}^{2}t+2\,q{t}^{2}+{t}^{3}+qt \right) s_{{1}}(\mbf{x})\\
&&+t \left( q{t}^{3}+{t}^{4}+{q}^{3}+{q}^{2}t+2\,q{t}^{2}+{t}^{3}+{q}^{2}+2\,qt+{t}^{2} \right) s_{{2}}(\mbf{x})\\
&&+( q{t}^{5}+{t}^{6}+{q}^{2}{t}^{3}+2\,q{t}^{4}+{t}^{5}+{q}^{4}+2\,{q}^{3}t+4\,{q}^{2}{t}^{2}\\
         &&\hskip4.5cm+4\,q{t}^{3}+2\,{t}^{4}+{q}^{2}t+q{t}^{2} )\, s_{{11}}(\mbf{x})\\
&&+ t\left( q+t \right) s_{{3}}(\mbf{x})+ \left( q+t \right)  \left( {t}^{3}+{q}^{2}+qt+2\,{t}^{2}+q+t\right) s_{{21}}(\mbf{x})\\
&&+ \left( {q}^{2}+qt+{t}^{2} \right)  \left( {q}^{3}+{t}^{3}+qt+q+t \right) s_{{111}}(\mbf{x})\\
&&+ \left( q+t \right) s_{{31}}(\mbf{x})+ \left( {q}^{2}+qt+{t}^{2} \right) s_{{22}}(\mbf{x})
\end{eqnarray*}
Among other interesting inequalities, we have
\begin{equation}\label{InclusionHook}
   \bleu{{q}\,(-1)^j s_{(j+1\,|\,k-1)}^{(a,b)}(\mbf{x};q,t)\quad {\preceq_s}\quad (-1)^{j-1}s_{(j\,|\,k)}^{(a,b)}(\mbf{x};q,t)},
 \end{equation}
for   two ``consecutive'' hooks\footnote{Notice that $j$ is the $\iota$-function value of the hook $(j\,|\,k)$.}. 
For instance, the inequalities
    $$-q^3 s_4^{(a,b)}(\mbf{x};q,t)\ \preceq_s\ q^2 s_{31}^{(a,b)}(\mbf{x};q,t)\ \preceq_s\  -q\, s_{211}^{(a,b)}(\mbf{x};q,t)\ \preceq_s\  s_{1111}^{(a,b)}(\mbf{x};q,t),$$
 correspond to the Schur-positive differences
 \begin{eqnarray*}
&&s_{{31}}^{(a,b)}(\mbf{x};q,t) -q\,s_{{4}}^{(a,b)}(\mbf{x};q,t)={t}^{2}s_{{22}}+ts_{{31}}+t \left( {t}^{2}+q+t
 \right) s_{{211}}+{t}^{2} \left( {t}^{2}+q \right) s_{{1111}}\\
&&s_{{211}}^{(a,b)}(\mbf{x};q,t)-q\,s_{{31}}^{(a,b)}(\mbf{x};q,t)={t}^{2}s_{{31}}+t \left( {t}^{2}+q \right) s_{{22
}}+{t}^{2} \left( {t}^{2}+q+t \right) s_{{211}}+{t}^{3} \left( {t}^{2}
+q \right) s_{{1111}}\\
&&s_{{1111}}^{(a,b)}(\mbf{x};q,t)-q\,s_{{211}}^{(a,b)}(\mbf{x};q,t)=s_{{4}}+ \left( {t}^{4}+q{t}^{2}+{q}^{2}+qt+{t}
^{2} \right) s_{{22}}\\
&&\hskip6cm+ \left( {t}^{3}+{q}^{2}+qt+{t}^{2}+q+t \right) s_{{31}}\\
&&\hskip6cm+ \left( {t}^{5}+q{t}^{3}+{t}^{4}+{q}^{3}+2\,{q}^{2}t+2\,q{t}^{
2}+{t}^{3}+qt \right) s_{{211}}\\
&&\hskip6cm+t \left( {t}^{5}+q{t}^{3}+{q}^{3}+{q}^
{2}t+q{t}^{2} \right) s_{{1111}}
 \end{eqnarray*}
 The compositional $(ad,bd)$-shuffle conjecture implies inequality~\pref{InclusionHook}. Indeed, we have the identity (shown in~\cite{IdPosCon}) 
   $$\bleu{(-q)^{j-1}\sum_{\alpha\vdash k} \mathbf{C}_j\mathbf{C}_\alpha(1) = 
             s_{(j\,|\,k)}(\mbf{x})+\frac{1}{q}s_{(j+1\,|\,k-1)}(\mbf{x})}.$$
This also shows that settling the compositional $(ad,bd)$-shuffle conjecture would answer in the affirmative the first part of \qref{question_schur} for any hook shapes (see~\cite[Proposition 5.3]{HMZ}). 

For all $(m,n)$, we have also observed (calculating all cases for $m,n\leq 8$) that the following inequality seems to hold 
  \begin{equation}\label{InclusionEH}
   \bleu{q^\beta\,e_{m-1,n}(\mbf{x};q,t)\ \preceq_s\ (-qt)^{1-d}h_{m,n}(\mbf{x};q,t) },
 \end{equation}
with $\beta=\alpha(m,n)-d+1$, for $d=\gcd(m,n)$. In other words, this is the number of points that lie between the diagonal avoiding $(m,n)$-staircase path, and the  $(m-1,n)$-staircase path. Once again, there seems to be a transpose version of this 
  \begin{equation}\label{InclusionEHco}
   \bleu{q^{\beta'}\,\overline{e_{m,n-1}(\mbf{x};q,t)}\ \preceq_s\ (-qt)^{1-d}\,\overline{h_{m,n}(\mbf{x};q,t)} },
 \end{equation}
 with $\beta'$ defined suitably.
Together with \pref{InclusionHook}, inequality~\pref{InclusionEH} refines the inequality in \pref{S_Inclusion}. Hence we are led to ask the following:
\begin{question}
Can we prove that, for all coprime $a,b\geq 1$, all $j$ and $k$, and all $m,n\geq 1$, inequalities~\pref{InclusionHook} and \pref{InclusionEH} hold?
\end{question}
\noindent As well as
\begin{question}
Can we prove that, for all pair $m,n\geq 1$, inequalities~\pref{Inclusionco} and \pref{InclusionEHco} hold?
\end{question}
\noindent
Preferably, these ``facts'' would be explained in terms of inclusion of bigraded representations.
Observe that, up to applying a sequence of such inclusions, we may include any of the relevant  expressions as subexpressions of $\nabla^a(e_n)=e_{an,n}(\mbf{x};q,t)$ (see~\pref{nabla_r_en})
which is conjectured to give the bigraded Frobenius characteristic of the $\S_n$-module $\mathcal{C}_n^{(a)}$ of the generalized diagonal coinvariant $\S_n$-module\footnote{Recall that the case $a=1$ has been shown to hold in~\cite{haiman}. }. Hence, Schur-positivity of the above differences would imply that we have bigraded-monomorphism between associated $\S_n$ modules, all of which included in $\mathcal{C}_n^{(a)}$, for $a$ large enough. Experiments suggest that these modules are ideals, generated by correctly chosen lowest  degree components.

\subsection*{Transpose sub-symmetry}
Following a somewhat different track, we have another kind of inclusion involving a matrix like ``transposition''. This seems to be a very general phenomenon that we have checked for all positive seeds considered here, as well as in the cases that correspond to the compositional $(ad,bd)$-shuffle conjecture (see~\cite{compositionalshuffle}). The most general question may be coined as follows:

\begin{question}
Can we prove that
   \begin{equation}\label{inclusiontranspose}
        \bleu{\overline{f_d^{(b,a)}(\mbf{x};q,t)} \preceq_s  \overline{f_d^{(a,b)}(\mbf{x};q,t)}},
    \end{equation}
 for all coprime $b\geq a\geq 1$, and any seed $f_d$ that gives rise to Schur positive expressions?
 \end{question}
 \noindent We underline that the functions $f_d^{(b,a)}(\mbf{x};q,t)$ and $f_d^{(a,b)}(\mbf{x};q,t)$ are of different degrees; equal to $ad$ in the first case, and $bd$ in the second. Hence they can only be compared after applying the $\overline{(-)}$ operator, which results in a symmetric function having components of various degrees. For example, we have
 \begin{eqnarray*}
    \overline{e_1^{(5,3)}(\mbf{x};q,t)}&=&\left( q+t \right) s_{2}+ \left( q+t \right) 
 \left( {q}^{2}+{t}^{2}+q+t \right) s_{1}\\
 &&\qquad\qquad 
+ \left( {q}^{4}+{q}^{3}t+{q}^{2}{t}^{2}+q{t}^{3}+{t}^{4}+{q}^{2}t+q{t}
^{2} \right) s_{0}\\
    \overline{e_1^{(3,5)}(\mbf{x};q,t)} &=&
 \left( q+t \right) s_{2}+ \left( q+t \right)  \left( {q}^{2}+{t}^{2}+q+t
 \right) s_{2}\\
 &&\qquad\qquad 
+ \left( {q}^{4}+{q}^{3}t+{q}
^{2}{t}^{2}+q{t}^{3}+{t}^{4}+{q}^{2}t+q{t}^{2} \right) s_{0}\\
  &&\quad + s_{21}+ \left( {q}^{2}+qt+{t}^{2}+q+t \right) s_{11}
\end{eqnarray*}
hence $ \overline{e_1^{(3,5)}(\mbf{x};q,t)}-\overline{e_1^{(5,3)}(\mbf{x};q,t)}=s_{21}+ \left( {q}^{2}+qt+{t}^{2}+q+t \right) s_{11}$.
 
As alluded to above, statement~\pref{inclusiontranspose} has been checked  by explicit computer algebra calculations for all cases involving  either $(-{qt})^{-\iota(\mu)}s_\mu^{(a,b)}(\mbf{x},q,t)$,  $(-1)^{d-\ell(\mu)}\, m_\mu^{(a,b)}(\mbf{x};q,t)$, or $C_\alpha^{(a,b)}(\mbf{x};q,t)$, for all partitions $\mu$, all compositions $\alpha$, and all coprime pairs $(a,b)$ for which the overall degree of the resulting function is at most $12$. Hence it holds for all situations that can be expressed as positive linear combinations of these.

\section{\texorpdfstring{$e$}{e}-Positivity and Specializations at \texorpdfstring{$t=1$, and $t=1+r$}{t}}
 Our next considerations concern an interesting feature of the specialization of the operators at $t=1$. Indeed, the resulting operators appear to be much simpler operators than their general counterpart. Indeed, one observes experimentally\footnote{This will be supported by actual results in the sequel.} that
  \begin{equation}\label{specializationat1}
      \bleu{\Theta_{a,b}(f_d)(g(\mbf{x}))\Big|_{t=1} = f_d^{(a,b)}(\mbf{x};q,1)\cdot g(\mbf{x})}.
  \end{equation}
This states that the effect of the operator $\Theta_{a,b}(f_d)\Big|_{t=1}$ on any symmetric function $g(\mbf{x})$ corresponds to multiplication of $g(\mbf{x})$ by the fixed symmetric function $ f_d^{(a,b)}(\mbf{x};q,1)$. In other words, at $t=1$, the monomorphism $\Theta_{a,b}$ may be considered as graded-algebra homomorphism
     $$\Theta_{a,b}\Big|_{t=1}:\bigoplus_{d\geq 0}\Lambda_d\longrightarrow \bigoplus_{d\geq 0}\Lambda_{bd}.$$
sending $f_d$ to (multiplication by) $f_d^{(a,b)}(\mbf{x};q,1)$. Notice here the ``multiplicative'' shift in grading, $d\mapsto bd$. Implicit in statement \pref{specializationat1} is the ``multiplicativity'' 
    \begin{equation}\label{multiplicativite}
      \bleu{ (f_d\,g_k)^{(a,b)}(\mbf{x};q,1)=f_d^{(a,b)}(\mbf{x};q,1)\,g_k^{(a,b)}(\mbf{x};q,1)}.
  \end{equation}  
Thus all of this would follow from the following:

\begin{question}\label{question_mult} 
Can we prove that the operator $\Theta_{a,b}(f_d)\big|_{t=1}$ operates by multiplication by $f_d^{(a,b)}(\mbf{x};q,1)$, for all seed $f_d$ and all coprime $a,b\geq 1$?
\end{question}
 \noindent Observe that it is clearly sufficient to answer this question for any given family of algebraic generators of $\Lambda$, say $\{q_d\}_{d\in\N}$ or $\{e_d\}_{d\in\N}$. Recall also from~\cite{IdPosCon} that $\widetilde{\nabla}$, the linear operator obtained from $\nabla$ by specializing $t$ to $1$, is multiplicative. Hence, we get the following.
 
\begin{proposition}
   If   $\Theta_{a,b}(f_d)\big|_{t=1}$ operates by multiplication by $f_d^{(a,b)}(\mbf{x};q,1)$, then $\Theta_{a+b,b}(f_d)\big|_{t=1}$ also operates by multiplication by  $f_d^{(a+b,b)}(\mbf{x};q,1)$.
  \end{proposition}
  \begin{proof}[\bf Proof]
Using~\pref{nabla} and \pref{nabla_fonct} specialized at $t=1$, we calculate that, for any symmetric function $g(\mbf{x})$,
     \begin{eqnarray*}
         \Theta_{a+b,b}(f_d)\big|_{t=1}(g(\mbf{x}))&=&\widetilde{\nabla} \Theta_{a,b}(f_d)\big|_{t=1}\widetilde{\nabla}^{-1} (g(\mbf{x}))\\
         &=&\widetilde{\nabla} \left[f_d^{(a,b)}(\mbf{x};q,1)\cdot \widetilde{\nabla}^{-1} (g(\mbf{x}))\right]\\
         &=&\widetilde{\nabla} (f_d^{(a,b)}(\mbf{x};q,1))\cdot\widetilde{\nabla} (\widetilde{\nabla}^{-1} (g(\mbf{x})))\\
         &=&f_d^{(a+b,b)}(\mbf{x};q,1)\cdot g(\mbf{x}),
    \end{eqnarray*}
    which shows the required property.
      \end{proof}
Observe also that, to answer Question-\ref{question_mult} positively in all instances, we need only show that $\Theta_{a,b}(e_d)\big|_{t=1}$ operates by multiplication. To this end, let us recall 
the following conjectured constant term formula of Negut (see~\cite{NegutShuffle}), 
   $$\bleu{\Theta_{a,b}(e_d)(g(\mbf{x}))=
   {\mathrm{CT}}\!\!\left(\frac{g[\mbf{x}+M\,\sum_{i=1}^m z_i^{-1}]}{\mbf{z}_{m,n}}\prod_{i=1}^{m-1}\frac{z_i}{z_i-qtz_{i+1}}\Omega'[\mbf{x};z_i]\!\!
   \prod_{1\leq i<j\leq m} \frac{(z_i-z_j)(z_i-qtz_j)}{(z_i-qz_j)(z_i-tz_j)} \right)}$$
for the calculation of the operators $\Theta_{a,b}(e_d)$, where the constant term is calculated with respect to the variables $\mbf{z}=z_1,\ldots,z_m$, and
     $$\bleu{\mbf{z}_{m,n}:= \prod_{i=1}^m z_i^{\lfloor i\,n/m\rfloor-\lfloor (i-1)\,n/m \rfloor}}.$$
  We use here the notation
      $$\bleu{\Omega'[\mbf{x};z]:=\sum_{n\geq 0} e_n(\mbf{x})\,z^n}$$
  for the \define{dual Cauchy kernel}\footnote{The Cauchy kernel $\Omega[\mbf{x};z]$, obtained by replacing $e_n(\mbf{x})$ replaced by $h_n(\mbf{x})$, is naturally related to the standard scalar product of symmetric functions.}.

Specializing at $t=1$ this constant term formula, one finds the  following further support for the ``fact'' that our operators have this multiplicative property  at $t=1$.
\begin{proposition}
Let $(m,n)$ be equal to $(ad,bd)$, with $d=\gcd(m,n)$, then Negut's conjecture implies that
    \begin{equation}\label{Eqoperator}
     \bleu{ \Theta_{a,b}(e_d)\big|_{t=1}(g(\mbf{x}))={\mathrm{CT}}\left(\frac{1}{\mbf{z}_{m,n}}\prod_{i=1}^{m-1}\frac{z_i}{z_i-qz_{i+1}}\Omega'[\mbf{x};z_i]\right)\cdot g(\mbf{x})},
 \end{equation}
  \end{proposition}
It is noteworthy that a combinatorial argument, discussed in~\cite{compositionalshuffle}, shows that the constant term involved in the right-hand side of~\pref{Eqoperator} corresponds to the enumeration of $(m,n)$-Dyck paths by area and risers, that is 
\begin{equation}\label{conjBGLX}
    \bleu{e_{m,n}(\mbf{x};q,1)={\mathrm{CT}}\left(\frac{1}{\mbf{z}_{m,n}}\prod_{i=1}^{m-1}\frac{z_i}{z_i-qz_{i+1}}\Omega'[\mbf{x}\,z_i]\right)=\sum_{\gamma} q^{\area(\gamma)}\,e_{\rho(\gamma)}(\mbf{x})}, 
 \end{equation}
with the sum running over the set of $(ad,bd)$-Dyck paths. One easily gets a similar constant term formula for the enumeration of  $(m,n)$-Dyck paths with no return to the diagonal, except at both ends. To this end, one simply replaces $\mbf{z}_{m,n}$ by $\mbf{z}_{m,n}/(z_1z_2\cdots z_m)$, and it corresponds (conjecturally) to the specialization at $t=1$ of a constant term formula for $(-q)^{1-d}h_d^{(a,b)}(\mathbf{x};q,t)$.

 Another interesting feature of this specialization at $t=1$ is made apparent for special seeds. Indeed, for these special cases, the symmetric function $ f_d^{(a,b)}(\mbf{x};q,1)$ appears to expand with coefficients in $\N[q]$ in the basis of elementary symmetric functions $e_\mu$, for $\mu$ partitions of $bd$. It is usual to say that they are \define{$e$-positive}. If $f$ is $e$-positive, we write $0\leq_{\rouge{e}} f$. This is clearly stronger than Schur-positivity, since it is classical that each $e_\mu$ is itself Schur-positive.
In fact, an even stronger version of $e$-positivity seems to be at play here, as stated by the following, which has been checked explicitly for all $j+k=d-1$, and all $a,b$ such that $1\leq ad,bd\leq 8$. 

\begin{question}\label{questione_e_pos}
Can we prove that 
\begin{eqnarray*}
       &&\bleu{0\leq_{\rouge{e}} (-1)^{1-k}h_d^{(a,b)}(\mbf{x};q,1+r)},\qquad{\rm and}\\
       &&
         \bleu{{q}\,(-1)^j s_{(j+1\,|\,k-1)}^{(a,b)}(\mbf{x};q,1+r)\ {\leq_{\rouge{e}}}\  (-1)^{j-1}s_{(j\,|\,k)}^{(a,b)}(\mbf{x};q,1+r)},
   \end{eqnarray*}
  for all $j+k=d-1$, and all coprime $a,b\geq 1$?
\end{question}
\noindent Exploiting the transitivity of the order, this implies that $(-1)^{j-1}s_{(j\,|\,k)}^{(a,b)}(\mbf{x};q,1+r)$ itself is $e$-positive, since $h_d(\mbf{x})=s_{(d-1|0)}(\mbf{x})$. This also implies (setting $r=0$) that 
      $$\bleu{0\leq_{\rouge{e}} e_d^{(a,b)}(\mbf{x};q,1)},\qquad {\rm and}\qquad \bleu{0\leq_{\rouge{e}} q_d^{(a,b)}(\mbf{x};q,1)},$$ 
in view of the definition of $q_d(\mbf{x};q,t)$.
 \noindent For example, for the seed $e_d(\mbf{x})$, some explicit values are
\begin{eqnarray*}
e_2^{(1,3)}(\mbf{x};q,1)&=&q^{3}\,e_6(\mbf{x})+q^{2}e_{51}(\mbf{x})+q\,e_{42}(\mbf{x})+e_{33}(\mbf{x}),\\
e_2^{(1,2)}(\mbf{x};q,1)&=&q^{6}\,e_6(\mbf{x})+q^{4}\, (q+1 )\, e_{51}(\mbf{x})+q^{2}\, (q^{2}+2 )\, e_{42}(\mbf{x})\\
\qquad &&+q^{3}e_{411}(\mbf{x})+q^{3}e_{33}(\mbf{x})+q\, (q+2 )\, e_{321}(\mbf{x})+e_{222}(\mbf{x}),\\
e_2^{(2,3)}(\mbf{x};q,1)&=&q^{8}\,e_6(\mbf{x})+q^{5}\, (q^{2}+q+1 )\, e_{51}(\mbf{x})+q^{4} (q^{2}+2 )\, e_{42}(\mbf{x})\\
\qquad &&+q^{3} (q^{2}+q+1 )\, e_{411}(\mbf{x})+q^{2} \,({q}^{3}+2\,{q}^{2}+2\,q+1)\, e_{33}(\mbf{x})\\
\qquad &&+q\, (q^{3}+3\,q^{2}+q+2)\, e_{321}(\mbf{x})+q^{2}e_{3111}(\mbf{x})\\
\qquad &&+q^{2}e_{222}(\mbf{x})+ (q+1 )\, e_{2211}(\mbf{x})^{2}.
\end{eqnarray*}
Now, as discussed in \cite{compositionalshuffle},  the $e$-positive symmetric functions $ f_d^{(a,b)}(\mbf{x};q,1)$ considered often appear to expand as a weighted sum, over combinatorial objects, of powers of $q$ multiplied by some elementary symmetric function, giving a combinatorial explanation why  they are {$e$-positive}. The relevant combinatorial objects are discussed in Appendix~A.

It is interesting to underline the following fact, which reduces the proof of $e$-positivity to the cases where $a\leq b$.
\begin{proposition}
   If $f_d^{(a,b)}(\mbf{x};q,1)$ is $e$-positive, then so is $f_d^{(a+b,b)}(\mbf{x};q,1)$. 
 \end{proposition}
\begin{proof}[\bf Proof] Recall from~\cite{IdPosCon} that on top of $\widetilde{\nabla}$ being is multiplicative, we have that $\widetilde{\nabla}(e_k)$ is $e$-positive. Hence, $\widetilde{\nabla}(e_\lambda)=\prod_{k\in\lambda} \widetilde{\nabla}(e_k)$ is $e$-positive for all $\lambda$, and we get the announced property since we get $f_d^{(a+b,b)}(\mbf{x};q,1)$ by applying $\widetilde{\nabla}$ to the $e$-positive expression $f_d^{(a,b)}(\mbf{x};q,1)$.
\end{proof}

Many other instances of $e$-positivity seem to occur, but they still have to be explained combinatorially. 
A tantalizing fact along these lines, discussed in~\cite[see Prop 2.3.4]{haiman}, is that the expression
$$\bleu{\langle p_1(\mbf{x})^n,e_{n,n}(\mbf{x},1,1+r)\rangle}$$
enumerates connected graphs, $r$-weighted by the number of edges. Extensive experiments, including all cases of degree $\leq 8$, lead to the following.

  \begin{question}\label{Sconjelem}
     Can we prove that, for any partition $\mu$ of $d$, and any coprime $a,b\geq 1$, that 
\begin{equation}\label{eposm}
   \bleu{0\leq_{\rouge{e}} (-1)^{d-\ell(\mu)}\, m_\mu^{(a,b)}(\mathbf{x};q,1+r)},
    \end{equation}
in other words, that the symmetric functions  are $e$-positive? Furthermore, can we explain this $e$-positivity in terms of a combinatorial enumeration in the style of \pref{conjBGLX}?
   \end{question} 
\noindent For example,  we have   
\begin{eqnarray*}
-m_{{21}}^{(1,1)}(\mathbf{x};q,1+r)&=&2\,e_1(\mathbf{x})^{3}+ \left( {q}^{2}r+q{r}^{2}+3\,{q}^{2}+4\,q\,r +2\,{r}^{2}+5\,q+6\,r \right) e_{{1}}(\mathbf{x})e_{{2}}(\mathbf{x})\\
&& +
 \left( {q}^{3}r +{q}^{2}{r}^{2}+q{r}^{3}+3\,{q}^{3}+3
\,{q}^{2}r +4\,q{r}^{2}+2\,{r}^{3}+5\,q\,r+4\,{r}^{2} \right) e_{{3}}(\mathbf{x}).
\end{eqnarray*}
Positive answers to these and \qref{question_mult} would imply many relations between $e$-positive expression.
For instance, one obtains Bizley-like formulas in the form
\begin{eqnarray}
   \bleu{\sum_{d\geq 0}  e_d^{(a,b)}(\mbf{x};q,1)\, z^d}&=&\bleu{\exp\!\Big( \sum_{j\geq 1} (-1)^{j-1} p_j^{(a,b)}(\mbf{x};q,1)\, z^j/j \Big)},\qquad {\rm and}\label{relBizley}\\
    \bleu{\sum_{k\geq 0}  h_d^{(a,b)}(\mbf{x};q,1)\, z^k}&=&\bleu{\exp\!\Big( \sum_{j\geq 1}  p_j^{(a,b)}(\mbf{x};q,1)\, z^j/j \Big)}.
 \end{eqnarray}   
Other interesting observations concern the compositional $(ad,bd)$-shuffle conjecture of~\cite{compositionalshuffle}, specialized at $t=1$. Indeed, as discussed in \cite{HMZ},  the evaluation at $1$ of the operator $\mbf{C}_\alpha$  specializes, at $t=1$, to the product of the $(-1)^{k-1}h_k$ with $k$ running over parts of $\alpha$, where $\alpha$ is a composition of $d$. Hence, modulo our above observations and if~\pref{shuffle_conjecture} holds, we should have
   \begin{equation}\label{retours}
     \bleu{C_\alpha^{(a,b)}(\mbf {x};q,1)=\sum_{\gamma} q^{\area(\gamma)} e_{\rho(\gamma)}},
   \end{equation}
 where $\gamma$ runs over the set of $(m,n)$-Dyck paths that return to the diagonal at the points 
     $$(a\,\alpha_i,b\,\alpha_i),\qquad {\rm for}\quad \alpha_i=k_1+\ldots+k_i,$$ 
with $i$ varying between $0$ and $\ell$. Thus, some cases that are common to \pref{eposm} and \pref{retours} are consequences of the \pref{shuffle_conjecture}, in particular this is so for $e_{m,n}(\mbf{x};q,1)$.

Using the combinatorial interpretation~\pref{retours}, we may readily see that the specialization at $t=1$ of \pref{S_Inclusion} and \pref{InclusionEH} hold. In fact, the relevant differences are in fact $e$-positive, since we have inclusion between the sets of paths enumerated by each expression. Hence we get the following.
\begin{proposition} For all $m$ and $n$,
\begin{eqnarray}
   \bleu{q^\alpha\,e_{m-1,n}(\mbf{x};q,1)} &\leq_{\rouge{e}}&\bleu {e_{m,n}(\mbf{x};q,1)},\qquad {\rm and}\\
   \bleu{q^\beta\,e_{m-1,n}(\mbf{x};q,1)} &\leq_{\rouge{e}}&\bleu {(-qt)^{1-d}h_{m,n}(\mbf{x};q,1)}.
\end{eqnarray}
    \end{proposition}
This raises the question of whether we have the stronger $e$-positivity property considered earlier in other instances, namely
\begin{question}\label{questione_e_posr}
Can we prove that 
\begin{eqnarray}
   \bleu{q^\alpha\,e_{m-1,n}(\mbf{x};q,1+r)} &\leq_{\rouge{e}}&\bleu {e_{m,n}(\mbf{x};q,1+r)},\qquad {\rm and}\\
   \bleu{q^\beta\,e_{m-1,n}(\mbf{x};q,1+r)} &\leq_{\rouge{e}}&\bleu {(-qt)^{1-d}h_{m,n}(\mbf{x};q,1+r)}.
\end{eqnarray}
  for all $m,n\geq 1$, and explain this combinatorially?
\end{question}\noindent
These statements have been explicitly checked to hold for all $m,n\leq 8$.

%

\section{Specialization at \texorpdfstring{$q=t=1$}{qt}}
We simplify our notation in this section, writing $f_d^{(a,b)}(\mbf{x})$ instead of $f_d^{(a,b)}(\mbf{x};1,1)$, and follow the logic of our previous conventions so that
   $$\bleu{q_\mu^{(a,b)}(\mbf{x}) := q_{\mu_1}^{(a,b)}(\mbf{x})q_{\mu_2}^{(a,b)}(\mbf{x})\cdots q_{\mu_\ell}^{(a,b)}(\mbf{x})}.$$
Once again we assume that $(m,n)=(ad,bd)$, with $(a,b)$ a coprime pair.
Then, an argument similar to that of \cite{bizley} (using \pref{relBizley} and \pref{conjBGLX}) shows that
\begin{eqnarray}
    \bleu{q_d^{(a,b)}(\mbf{x})}
       &=&\bleu{(-1)^{d-1}p_d^{(a,b)}(\mbf{x})}\\
       &=&\bleu{\frac{d}{m}\, e_n[m\,\mbf{x}]  =  \frac{1}{a} e_{db}[da\,\mbf{x}] },\label{formule_q}
\label{Qchar}
    \end{eqnarray}
    so that, using the multiplicativity property~\pref{multiplicativite},
\begin{equation}
    \bleu{q_\mu^{(a,b)}(\mbf{x})  = \prod_{k\in\mu} \frac{1}{a} e_{kb}[ka\,\mbf{x}]}.
 \end{equation}
From this it follows  that we have a generalized Bizley-like formula
\begin{eqnarray}
   \bleu{f_d^{(a,b)}(\mbf{x})}
       &=& \bleu{\sum_{\mu\vdash d} f_\mu\, q_\mu^{(a,b)}(\mbf{x})}\nonumber\\
       &=& \bleu{\sum_{\mu\vdash d} f_\mu\, \prod_{k\in\mu} \frac{1}{a} e_{kb}[ka\,\mbf{x}]},\label{formule}
 \end{eqnarray}
if we have the expansion
    $$\bleu{f_d(\mbf{x}) = \sum_{\mu\vdash d} f_\mu\,q_\mu(\mbf{x})}.$$
 For example,  
 \begin{eqnarray*}
     e_3^{(a,b)}(\mbf{x}) &=& \frac{1}{3} q_3^{(a,b)}(\mbf{x})+\frac{1}{2}q_2^{(a,b)}(\mbf{x})\,q_1^{(a,b)}(\mbf{x})+
                   \frac{1}{6} q_1^{(a,b)}(\mbf{x})^3,\\
                &=& \frac{1}{3\,a}\, e_{{3\,b}} [ 3\,a\,\mbf{x} ]
                  +\frac{1}{2\,a^2}\,e_{{2\,b}} [ 2\,a\,\mbf{x} ] \,e_{{b}} [ a\,\mbf{x} ] 
                  +\frac {1}{6\,a^3} \left( e_{{b}} [a\,\mbf{x}]  \right) ^{3},\\
       -s_{21}^{(a,b)}(\mbf{x})&=& \frac{1}{3} q_3^{(a,b)}(\mbf{x})-\frac{1}{3} q_1^{(a,b)}(\mbf{x})^3,\\
                   &=& \frac{1}{3\,a}\, e_{{3\,b}} [ 3\,a\,\mbf{x} ] 
                  -\frac{1}{3\,a^3} \left( e_{{b}} [ a\,\mbf{x} ]  \right) ^{3},\\
     h_3^{(a,b)}(\mbf{x})&=&\frac{1}{3} q_3^{(a,b)}(\mbf{x})-\frac{1}{2}q_2^{(a,b)}(\mbf{x})q_1^{(a,b)}(\mbf{x})+
                   \frac{1}{6} q_1^{(a,b)}(\mbf{x})^3,\\
                   &=&\frac{1}{3\,a}\, e_{{3\,b}} [ 3\,a\,\mbf{x} ] 
                  -\frac{1}{2\,a^2}\, e_{{2\,b}} [ 2\,a\,\mbf{x} ]\,e_{{b}} [ a\,\mbf{x} ]  
                  +\frac {1}{6\,a^3} \left( e_{{b}} [ a\,\mbf{x} ]  \right) ^{3}.
 \end{eqnarray*}
Let us now consider the
linear transformations on symmetric functions
\begin{eqnarray}
      \bleu{\delta(g(\mbf{x}))}&:=&\bleu{\langle p_1(\mbf{x})^n,g(\mbf{x})\rangle },\\
     \bleu{\varepsilon(g(\mbf{x}))}&:=&\bleu{\langle e_n(\mbf{x}),g(\mbf{x})\rangle },
  \end{eqnarray}
for which we clearly have
\begin{eqnarray*}   
    \bleu{\delta(g_{d_1}(\mbf{x}))g_{d_2}(\mbf{x}))\cdots g_{d_\ell}(\mbf{x})) }&=&\bleu{\binom{n}{d_1,d_2,\ldots,d_\ell} \prod_{i=1}^\ell
           \delta(g_{d_i}(\mbf{x}))},\qquad {\rm and}\\
      \bleu{\varepsilon(g_{d_1}(\mbf{x}))g_{d_2}(\mbf{x}))\cdots g_{d_\ell}(\mbf{x})) }&=&\bleu{\prod_{i=1}^\ell
           \varepsilon(g_{d_i}(\mbf{x}))},
   \end{eqnarray*}
   where $n=d_1+d_2+\ldots+d_\ell$.
 Also recall that 
\begin{eqnarray*}
     \bleu{\delta(f_d^{(a,b)}(\mbf{x}))}&=&\bleu{\dim(M^{(a,b)}_{f_d})},\qquad {\rm and}\\
      \bleu{\varepsilon(f_d^{(a,b)}(\mbf{x}))}&=&\bleu{\dim\left(M^{(a,b)}_{f_d}\right)^\pm},
    \end{eqnarray*}
 whenever $f_d^{(a,b)}(\mbf{x})$ may be interpreted as the Frobenius characteristic of some ${\mathbb S}_n$-module $M^{(a,b)}_{f_d}$,
 with $(M^{(a,b)}_{f_d})^\pm$ standing for the alternating isotypic component of this ${\mathbb S}_n$-module.
 Since
     $$\bleu{\delta(q_d^{(a,b)}(\mbf{x})) =d\,m^{n-1}= d^{bd} a^{bd-1}},$$
 and
     $$\bleu{\varepsilon(q_d^{(a,b)}(\mbf{x}))=\frac{d}{m+n}\,\binom{n+m}{n}=\frac{1}{a+b}\binom{(a+b)d}{bd}},$$
 for any partition $\mu$ of $d$, with $(m,n)=(ad,bd)$ and $d=\gcd(m,n)$ as before, we have
\begin{eqnarray*}
    \bleu{\delta(q_\mu^{(a,b)}(\mbf{x}))}&=&\bleu{ \binom{n}{a\mu} a^{n-\ell(\mu)}\prod_{k\in\mu} k^{bk}},\\
    \bleu{\varepsilon(q_\mu^{(a,b)}(\mbf{x}))}&=&\bleu{\frac{1}{(a+b)^{\ell(\mu)}}\prod_{k\in\mu} \binom{(a+b)k}{bk}},
 \end{eqnarray*}  
where we use\footnote{Observe that $a\mu$ is a partition of $n$, with parts $a\mu_i$.} the partition multinomial notation
  $$\bleu{\binom{n}{a\,\mu} := \frac{n!}{(a\,\mu_1)! \cdots (a\,\mu_\ell)!}}.$$
  Thus, for $M^{(a,b)}_{f_d}$ the be the required $\S_n$-module would have to have the dimension formulas
\begin{eqnarray}
    \bleu{\dim(M^{(a,b)}_{f_d})}&=&\bleu{\sum_{\mu\vdash d} f_\mu \binom{n}{a\,\mu} a^{n-\ell(\mu)}\prod_{k\in \mu} k^{kb-1}}, \qquad {\rm and}\\
    \bleu{\dim(M^{(a,b)}_{f_d})^\pm}&=&\bleu{\sum_{\mu\vdash d} \frac{ f_\mu} {(a+b)^{\ell(\mu)}}\prod_{k\in\mu} \binom{(a+b)k}{bk}},
    \label{e_multiplicites}
 \end{eqnarray}
with the coefficients $f_\mu$ coming from the expansion \pref{formule}.
Observe that, in view of the dual Cauchy formula, the right-hand side of \pref{formule_q} affords a positive integer coefficient expansion in the $e$-basis given by the formula
\begin{eqnarray*}
     \bleu{\frac{d}{m}\, e_n[m\,\mbf{x}]} 
          &=&
          \bleu{\sum_{\lambda\vdash n} e_\lambda(\mbf{x})\,\frac{d}{m}\,h_\lambda[m]}\\
          &=&
          \bleu{\sum_{\lambda\vdash n} e_\lambda(\mbf{x})\,\frac{d}{m}\prod_{k\in \lambda}\binom{m+k-1}{k}},
 \end{eqnarray*}
 with $d=\gcd(m,n)$ as before.
Recalling that $\langle e_n(\mbf{x}),e_\lambda(\mbf{x})\rangle=1$ for all partition $\lambda$ of $n$, it follows that the sum of the coefficients of \pref{formule}, when expanded in the $e$-basis, must be equal to the number of copies of the alternating representations in $M_{f_d}$. In other words, it is the dimension of $(M^{(a,b)}_{f_d})^\pm$, as given by \pref{e_multiplicites}.

 \subsection*{Other specializations}
Some other possibilities of specializing $q$ and $t$ have been considered in the ``classical'' context of $e_{n,n}(\mbf{x};q,t)$, and then taking scalar product with $p_1^n$. For instance, in ~\cite{postnikov}, the authors set $t=-1$ and $q=1$, for which they get interesting combinatorial considerations. A similar specialization, followed by a scalar product with $p_1^n$, seems to give rise to many interesting combinatorial questions when one considers $f_d^{(a,b)}(\mbf{x};q,t)$ for seeds such as considered here.

\section*{Appendix A: Combinatorics of \texorpdfstring{$(m,n)$-Dyck}{mnDyck}  paths}
Recall that an \define{$(m,n)$-Dyck paths}  is a south-east lattice path, going from $(0,n)$ to $(m,0)$, which stays above the \define{$(m,n)$-diagonal}. This is the line segment joining $(0,n)$ to $(m,0)$. See Figure~\ref{fig1} for an example. 
\begin{figure}[ht]
\setlength{\unitlength}{4mm}
\setlength{\carrelength}{4mm}
\def\jcarre{\put(0,0){\jaune{\linethickness{\carrelength}\line(1,0){1}}}}
\def\palecarre{\bleupale{\linethickness{\unitlength}\line(1,0){1}}}
\begin{center}
\begin{picture}(11,5)(0,0)
\put(0,.5){\multiput(0,2)(1,0){3}{\jcarre}
                   \multiput(0,1)(1,0){6}{\jcarre}
                   \multiput(0,0)(1,0){7}{\jcarre}}
\multiput(0,0)(0,1){6}{\line(1,0){10}}
\multiput(0,0)(1,0){11}{\line(0,1){5}}
\thicklines
\put(-.6,5.5){$\scriptscriptstyle(0,5)$}
\put(10,-.6){$\scriptscriptstyle(10,0)$}
 \put(0,5){\bleu{\line(2,-1){10}}}
  \linethickness{.5mm}
\put(0,5){\rouge{\line(0,-1){1}}}\put(-1,0.2){$7$}
\put(0,4){\rouge{\line(0,-1){1}}}\put(-1,1.2){$6$}
\put(0,3){\rouge{\line(1,0){3}}}
\put(3,3){\rouge{\line(0,-1){1}}}\put(-1,2.2){$3$}
\put(3,2){\rouge{\line(1,0){3}}}
\put(6,2){\rouge{\line(0,-1){1}}}\put(-1,3.2){$0$}
\put(6,1){\rouge{\line(1,0){1}}}
\put(7,1){\rouge{\line(0,-1){1}}}\put(-1,4.2){$0$}
\put(7,0){\rouge{\line(1,0){3}}}
\end{picture}\end{center}
\caption{The $(10,5)$-Dyck path encoded as $00367$.}
\label{fig1}
\end{figure}
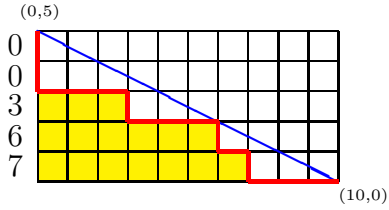

We encode such paths as (weakly) increasing integer sequences (words)
   $$\gamma=a_1a_2\cdots a_n,\qquad {\rm with}\quad 0\leq a_k\leq (k-1)\,m/n.$$
 Each $a_k$ gives the distance between the $y$-axis of the (unique) south step that starts at level $n+1-k$.  If $d=\gcd(m,n)$, we may only have equality $a_k=(k-1)\,m/n$ when $k=j\,b+1$, for $0<j<d$. If this is the case, we say that we have a \define{return} to the diagonal at position $k$. The set of such return positions may be encoded as a composition of $d$. This uses the classical correspondence between subsets of $\{1,\ldots,d-1\}$ and compositions $\alpha$ of $d$. To a composition $\alpha=(c_1,\ldots,c_k)$ this correspondence associates the set of partial sums $S(\alpha)=\{s_1,s_2,\ldots,s_k\}$, where
    $$s_i=c_1+c_2 +\cdots+c_i,\qquad {\rm with}\qquad 1\leq i<k.$$
The $(m,n)$-Dyck that stays ``closest'' to the diagonal is called the \define{$(m,n)$-staircase path} 
   \begin{equation}\label{defstaircase}
      \bleu{\delta_{m,n}:=d_1d_2\cdots d_n},\qquad{\rm with}\qquad  \bleu{d_k:= \lfloor (k-1)\,m/n\rfloor}.
   \end{equation} 
 For example, we have
$$\begin{array}{rclrclrclrcl}
\delta_{{1,4}}=0000,& \delta_{{2,4}}=0011,& \delta_{{3,4}}=0012,& \delta_{{4,4}}=0123,\\[4pt]
\delta_{{5,4}}=0123,& \delta_{{6,4}}=0134,& \delta_{{7,4}}=0135,& \delta_{{8,4}}=0246,\\[4pt]
\delta_{{9,4}}=0246,& \delta_{{10,4}}=0257,& \delta_{{11,4}}=0258,& \delta_{{12,4}}=0369.
\end{array}$$
It is easy to check that $\delta_{kn,n}=\delta_{kn+1,n}$. We denote by $\Dyck{m}{n}$, the set of $(m,n)$-Dyck paths, and by $\Cat{m}{n}$ its cardinality. For example, we have 
 $$\Dyck{5}{4}=\{ 0000, 0001, 0002, 0003, 0011, 0012, 0013, 0022, 0023, 0111, 0112, 0113, 0122, 0123\}.$$
 It follows from the observation that $\delta_{kn,n}=\delta_{kn+1,n}$, that we have the set equality
\begin{equation}\label{observation}
   \bleu{\Dyck{kn}{n}=\Dyck{kn+1}{n}}.
 \end{equation}
When $m$ and $m$ are coprime, the enumeration of $(m,n)$-Dyck path is given by the ``well'' known formula  \begin{displaymath}
   \bleu{\Cat{m}{n} =\frac{1}{m+n}\binom{m+n}{n}}.
\end{displaymath}
For the more general situation, when $m$ and $n$ have greatest common divisor $d\geq 1$, the formula was obtained by Bizley~\cite{bizley} in 1954. His argument may be given a more general understanding, using a symmetric function encoding of 
the multiplicities of parts in $(m,n)$-Dyck paths. To this end, we consider the \define{riser composition} $\rho(\gamma)$ of a path $\gamma$, which is simply the sequence of multiplicities of the entries of $\gamma$. We may then count $(m,n)$-paths with weight $e_{\rho(\gamma)}(\mathbf{x}):=e_{r_1}(\mathbf{x})e_{r_2}(\mathbf{x})\cdots e_{r_k}(\mathbf{x})$, if $\rho(\gamma)=r_1r_2\cdots r_k$.

Let $(m,n)=(ad,bd)$, with $a$ and $b$ coprime. It may be shown that (see~\cite{compositionalshuffle})
\begin{displaymath}
     	   \bleu{q_d^{(a,b)}(\mathbf{x};1,1):= \frac{d}{m}\,e_{n}[m,\mathbf{x}]},
\end{displaymath} 
in which one considers $m$ as a constant\footnote{This means that $p_k[m\mbf{x}]=m\,p_k(\mbf{x})$.} for the pletystic evaluation of the right-hand side.  
Then, a symmetric function version of Bizley's formula may be written as
  \begin{equation}\label{bizley_formula}
   \bleu{\sum_{\mu\vdash d} q_\mu^{(a,b)}(\mathbf{x};1,1)/{z_\mu} = \sum_{\gamma\in\Dyck{ad}{bd}} e_{\rho(\gamma)}(\mathbf{x})
             }.
\end{equation}
Recall that, for a partition $\mu$ of $d$ having $c_i$ parts of size $i$, the integers $z_\mu$ are defined as
		\begin{displaymath} \bleu{z_\mu:=\prod_k k^{c_k}\, c_k!}\end{displaymath}
Expressed in generating function terms, formula~\ref{bizley_formula}  takes the form
\begin{equation}\label{bizley_gen}
   \bleu{ \sum_{d=0}^\infty \sum_{\gamma\in\Dyck{ad}{bd}} e_{\rho(\gamma)}(\mathbf{x})\,x^d = \exp\!\left(\sum_{k\geq 1} \frac{1}{a}\,e_{bk}[ak\,\mathbf{x}] \frac{x^k}{k}\right)}.
\end{equation} 
For example, we have
\begin{eqnarray*}
  \bleu{\sum_{\gamma\in\Dyck{2a}{2b}} e_{\rho(\gamma)}(\mathbf{x})}&=& \bleu{\frac{1}{2} \left(\frac{1}{a}e_{b}[a\,\mathbf{x}]\right)^2+\frac{1}{2}\left(\frac{1}{a}e_{2b}[2a\,\mathbf{x}]\right)},\\[4pt]
  \bleu{\sum_{\gamma\in\Dyck{3a}{3b}} e_{\rho(\gamma)}(\mathbf{x})}&=&\bleu{\frac{1}{6}\left(\frac{1}{a}e_{b}[a\,\mathbf{x}]\right)^3
  +\frac{1}{2}\left(\frac{1}{a}e_{b}[a\,\mathbf{x}]\right)\left(\frac{1}{a}e_{2b}[2a\,\mathbf{x}]\right)}\\
  &&\qquad \qquad\bleu{+\frac{1}{3}\left(\frac{1}{a}e_{3b}[3a\,\mathbf{x}]\right).
}
\end{eqnarray*}
One obtains Bizley's formula as the coefficient of $e_n(\mathbf{x})$ in the resulting elementary symmetric function expansion.
Bizley also obtained a formula for the number of \define{primitive} $(ad,bd)$-Dyck paths. These are the paths that remain strictly above the diagonal (except at both ends). The symmetric function enumerator for these is
  \begin{eqnarray}\label{bizley_primitif}
     	   \bleu{h_d(\mbf{x};1,1)}&=&\bleu{\sum_{\mu\vdash d}{p^{(a,b)}_\mu(\mathbf{x};1,1)}/{z_\mu}}\nonumber\\
	      &=&\bleu{\frac{1}{a}\,h_{bk}[ak\,\mathbf{x}]}.
\end{eqnarray} 
From this, we may easily enumerate $(m,n)$-Dyck paths with specified return positions to the diagonal.

\subsection*{Area of \texorpdfstring{$(m,n)$}--Dyck paths}
The \define{area} of an $(m,n)$-Dyck path $\alpha$ is the number of \define{cells}\footnote{These are the $1\times 1$ squares in the $\N\times \N$-grid, and they are labeled by their southwest corner.} lying entirely between the path $\alpha$ and the $(m,n)$-staircase:
    \begin{equation}\label{defn_area}
          \bleu{\area_{m,n}(\alpha)}:=\bleu{\sum_{i=k}^n d_k-a_k},
      \end{equation}
 where the $\delta_{m,n}=d_1\cdots d_n$ is the $(m,n)$-staircase. 
In particular, $\delta_{m,n}$ is the unique $(m,n)$-Dyck path having area zero. 

\setlength{\unitlength}{4mm}
\newdimen\carrelength
\setlength{\carrelength}{3.5mm}
\begin{figure}[ht]
$$\begin{array}{ccccc}
     \begin{picture}(4,3)(0,0)
\multiput(0,1)(0,1){3}{\line(1,0){3}}
\multiput(1,0)(1,0){3}{\line(0,1){3}}
 \put(0,0){\line(1,0){3}}
 \put(0,0){\line(0,1){3}}
  \put(0,3){\bleu{\line(1,-1){3}}}
  \linethickness{.5mm}
\put(0,3){\rouge{\line(0,-1){1}}}
\put(0,2){\rouge{\line(1,0){1}}}
\put(1,2){\rouge{\line(0,-1){1}}}
\put(1,1){\rouge{\line(1,0){1}}}
\put(2,1){\rouge{\line(0,-1){1}}}
\put(2,0){\rouge{\line(1,0){1}}}
\end{picture}        &
     \begin{picture}(4,3)(0,0)
     \multiput(1.1,0.5)(1,0){1}{\jcarre}
\multiput(0,1)(0,1){3}{\line(1,0){3}}
\multiput(1,0)(1,0){3}{\line(0,1){3}}
 \put(0,0){\line(1,0){3}}
 \put(0,0){\line(0,1){3}}
  \put(0,3){\bleu{\line(1,-1){3}}}
  \linethickness{.5mm}
\put(0,3){\rouge{\line(0,-1){1}}}
\put(0,2){\rouge{\line(1,0){1}}}
\put(1,2){\rouge{\line(0,-1){2}}}
\put(1,0){\rouge{\line(1,0){2}}}
\end{picture}        &
     \begin{picture}(4,3)(0,0)
     \multiput(0.1,1.5)(1,0){1}{\jcarre}
\multiput(0,1)(0,1){3}{\line(1,0){3}}
\multiput(1,0)(1,0){3}{\line(0,1){3}}
 \put(0,0){\line(1,0){3}}
 \put(0,0){\line(0,1){3}}
  \put(0,3){\bleu{\line(1,-1){3}}}
  \linethickness{.5mm}
\put(0,3){\rouge{\line(0,-1){2}}}
\put(0,1){\rouge{\line(1,0){2}}}
\put(2,1){\rouge{\line(0,-1){1}}}
\put(2,0){\rouge{\line(1,0){1}}}
\end{picture}       &
     \begin{picture}(4,3)(0,0)
          \multiput(1.1,0.5)(1,0){1}{\jcarre}
         \multiput(0.1,1.5)(1,0){1}{\jcarre}
\multiput(0,1)(0,1){3}{\line(1,0){3}}
\multiput(1,0)(1,0){3}{\line(0,1){3}}
 \put(0,0){\line(1,0){3}}
 \put(0,0){\line(0,1){3}}
  \put(0,3){\bleu{\line(1,-1){3}}}
  \linethickness{.5mm}
\put(0,3){\rouge{\line(0,-1){2}}}
\put(0,1){\rouge{\line(1,0){1}}}
\put(1,1){\rouge{\line(0,-1){1}}}
\put(1,0){\rouge{\line(1,0){2}}}
\end{picture}        &
     \begin{picture}(4,3)(0,0)
        \multiput(0.1,0.5)(1,0){2}{\jcarre}
         \multiput(0.1,1.5)(1,0){1}{\jcarre}
\multiput(0,1)(0,1){3}{\line(1,0){3}}
\multiput(1,0)(1,0){3}{\line(0,1){3}}
 \put(0,0){\line(1,0){3}}
 \put(0,0){\line(0,1){3}}
  \put(0,3){\bleu{\line(1,-1){3}}}
  \linethickness{.5mm}
\put(0,3){\rouge{\line(0,-1){3}}}
\put(0,0){\rouge{\line(1,0){3}}}
\end{picture}\\
0 & 1 & 1 & 2 & 3
       \end{array}$$
       \vskip-10pt
       \caption{The areas of $(3,3)$-Dyck paths.}
\label{qcat3}
\end{figure}
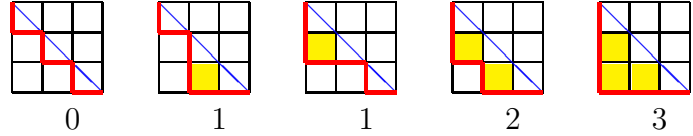

\subsection*{Parking functions, and dinv statistic}
An $(m,n)$-parking function is simply a permutation of the entries of an $(m,n)$-Dyck path. It may be represented as a labeling of the south steps of the path. To this end, a step is labeled $i$ if the corresponding entry appears in the $i^{\rm th}$-position  in a parking function $\pi$. If this step starts at $(x,y)$, we write $\pi(x,y)=i$. In other words, $i$ appears in the cell having coordinates $(x,y)$. This is illustrated in Figure~\ref{fig_park}, for the parking functions such that $\pi(0,0)=2$, $\pi(0,1)=4$, $\pi(3,2)=3$, $\pi(6,3)=1$, and $\pi(7,4)=5$.
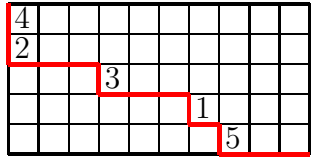
\begin{figure}[ht]
\begin{center}
\begin{picture}(11,5)(0,0)
\multiput(0,1)(0,1){5}{\line(1,0){10}}
\multiput(1,0)(1,0){10}{\line(0,1){5}}
\thicklines
 \put(0,0){\line(1,0){10}}
 \put(0,5){\line(1,0){10}}
 \put(0,0){\line(0,1){5}}
 \put(10,0){\line(0,1){5}}
  \linethickness{.5mm}
\put(0,5){\rouge{\line(0,-1){1}}\put(0.2,-.8){4}}
\put(0,4){\rouge{\line(0,-1){1}}\put(0.2,-.8){2}}
\put(0,3){\rouge{\line(1,0){3}}}
\put(3,3){\rouge{\line(0,-1){1}}\put(0.2,-.8){3}}
\put(3,2){\rouge{\line(1,0){3}}}
\put(6,2){\rouge{\line(0,-1){1}}\put(0.2,-.8){1}}
\put(6,1){\rouge{\line(1,0){1}}}
\put(7,1){\rouge{\line(0,-1){1}}\put(0.2,-.8){5}}
\put(7,0){\rouge{\line(1,0){3}}}
\end{picture}\end{center}
\caption{The $(10,5)$-parking function $60307$.}
\label{fig_park}
\end{figure}

As illustrated in Figure~\ref{table_rank}, the $(m,n)$-\define{rank} of a cell $(x,y)$ is defined as being equal to $\rank(x,y):=n\,m-y\,m-x\,n$. 
\begin{figure}[ht]
$$
\begin{array} {rrrrrrrrrr}
\vdots &\vdots&\vdots&\vdots&\vdots&\vdots&\vdots&\vdots\\ \noalign{\medskip}
\rouge{0}&-5&-10&-15&-20&-25&-30&-35&\cdots\\ \noalign{\medskip}
\bleu{7}&\rouge{2}&-3&-8&-13&-18&-23&-28&\cdots\\ \noalign{\medskip}
14&\bleu{9}&\rouge{4}&-1&-6&-11&-16&-21&\cdots\\ \noalign{\medskip}
21&16&\bleu{11}&\bleu{6}&\rouge{1}&-4&-9&-14&\cdots\\ \noalign{\medskip}
28&23&18&13&\bleu{8}&\rouge{3}&-2&-7&\cdots\\ \noalign{\medskip}
35&30&25&20&15&\bleu{10}&\bleu{5}&\rouge{0}&\cdots
\end {array}
$$
\caption{Examples of $(m,n)$-ranks (with $m=7$ and $n=5$).}\label{table_rank}
\end{figure}
The \define{descent set} $\mathrm{des}(\pi)$ of a parking function $\pi$ is the set of $i$ ($<n$) for which $i+1$ sits in a cell of lower or equal rank to that of the cell in which $i$ appears, hence
    $$ \mathrm{des}(\pi):=\{ i\ |\ \pi(x,y)=i,\ \pi(u,v)=i+1,\  \rank(x,y)\geq \rank(u,v)\}.$$
We write $\comp(\pi)$ for the composition of $n$ that encodes this subset of $\{1,\ldots,n-1\}$. In the next section, we will need to consider composition indexed Schur functions. These are obtained by extending to compositions the classical Jacobi-Trudi formula. More explicitly, for a composition $\alpha=(c_1,\cdots,c_k)$, one sets
   $$\bleu{s_\alpha(\mathbf{x}):=\det(h_{c_i-i+j}(\mathbf{x}))_{1\leq i,j\leq k}}.$$
   It may easily be seen that this evaluates either to $0$, or to a single Schur function up to a sign.

\end{document}